\documentclass{amsart}
\usepackage{amsmath}
\usepackage{amssymb}
\usepackage{latexsym}
\usepackage{amsbsy}
\usepackage[all]{xy}
\usepackage{amscd}
\usepackage[dvips]{graphicx}

\newtheorem{theorem}{Theorem}[section]
\newtheorem{lemma}[theorem]{Lemma}

\theoremstyle{definition}

\newtheorem{example}[theorem]{Example}
\newtheorem{Prop}[theorem]{Proposition}
\newtheorem{Cor}[theorem]{Corollary}

\numberwithin{equation}{section}

\newcommand \bd{{\bf d}}
\newcommand\bbn{{\mathbb N}}
\newcommand\bbz{{\mathbb Z}}
\newcommand\bbq{{\mathbb Q}}

\newcommand\bbc{{\mathbb C}}

\newcommand\ra{\rightarrow}
\newcommand\aut{\mbox{Aut}}

\newcommand\ext{\mbox{Ext}\,}

\newcommand\Hom{\mathrm{Hom}}

\newcommand\udim{\mbox{\underline {dim}}}
\newcommand\ed{\mbox{End}}

\newcommand\mc{{\mathcal{C}}}

\newcommand\md{{\mathcal{D}}}

\newcommand\ue{{\underline{e}}}
\newcommand\uee{{\underline{\bf e}}}

\begin{document}

\title{Remarks on Hall algebras of triangulated categories}

\author{Jie Xiao}
\author{Fan Xu}

\address{Department of Mathematics\\ Tsinghua University,
Beijing{\rm 100084},P.R.China} \email{jxiao@math.tsinghua.edu.cn}

\address{Department of Mathematics\\ Tsinghua University,
Beijing{\rm 100084},P.R.China} \email{fanxu@mail.tsinghua.edu.cn}

\thanks{Jie Xiao was supported by NSF of China (No. 11131001) and Fan Xu was supported by NSF of China (No. 11071133).}

%    General info
\subjclass[2000]{Primary  16G20, 17B67; Secondary  17B35, 18E30}

\date{\today}

\keywords{derived Hall algebra, motivic Hall algebra.}

\begin{abstract}
We introduce the notion of the Drinfeld dual of an algebra and show that Hall algebras defined by Kontsevich-Soibelman in \cite{KS} are the Drinfeld duals of derived Hall algebras defined in \cite{Toen2005} and \cite{XX2006}. Moreover, we construct the motivic analogue of a derived Hall algebra and prove
that it is isomorphic to the motivic Hall algebra constructed in \cite{KS}.
\end{abstract}

\maketitle

\section{Introduction}
Let $k$ be a finite field with $q$ elements and {$\mathcal{A}$} be a
finitary $k$-category, i.e., a (small) $k$-linear abelian category satisfying:
$(1)$ $\mathrm{dim}_k\Hom_{\mathcal{A}} (M,N) <\infty $; $(2)$
$\mathrm{dim}_k\ext^1_{\mathcal{A}} (M,N) <\infty$ for any $M,N \in
\mathcal{A}$. The Hall algebra  $\mathcal{H}(\mathcal{A})$ associated
to a finitary category $\mathcal{A}$ is originally defined by Ringel \cite{Ringel1990} in order to realize a quantum group. In the
simplest version, it is an associative algebra , which, as a
$\bbq$-vector space, has a basis consisting of the isomorphism
classes $[X]$ for $X\in \mathcal{A}$ and has the multiplication
$[X]*[Y]=\sum_{[L]}\mathrm{g}_{XY}^L[L]$,
 where $X,Y,L \in
\mathcal{A}$ and $\mathrm{g}_{XY}^L=|\{ M \subset L| M \simeq X
\text{ and } L/M \simeq Y\}|.$ The structure constant $\mathrm{g}_{XY}^L$ is called  \emph{Hall number} and the algebra  $\mathcal{H}(\mathcal{A})$ now is called Ringel-Hall algebra. The Ringel-Hall algebras have been developed
many variants (see \cite{Joyce2}) as a framework involving the
categorification and the geometrization of Lie algebras and quantum
groups in the past two decades (for example, see \cite{Ringel1990,
Lusztig, Rie1994, Nakajima1998, PX1997, PX2000}). If $\mathcal{A}=\mathrm{mod}\Lambda$ for a hereditary finitary $k$-algebra $\Lambda$, then there exists a comultiplication $\delta: \mathcal{H}(\mathcal{A})\rightarrow \mathcal{H}(\mathcal{A})\otimes\mathcal{H}(\mathcal{A})$ constructed by Green (\cite{Green}). Burban and Schiffmann
also studied the (topological) comultiplication of $\mathcal{H}(Coh\mathbb{X})$ for some curves $\mathbb{X}$ (\cite{BS1}, \cite{BS2}). The comultiplication by Green
naturally induces a new multiplication on $\mathcal{H}(\mathcal{A})$. We call this new algebra structure the Drinfeld dual of $\mathcal{H}(\mathcal{A})$ (Section 2).

In \cite{Toen2005}, To\"en introduced derived Hall algebras associated to derived categories. In \cite{XX2006}, the notion was extended for triangulated categories with some homological finiteness conditions and a new proof for the associativity of derived Hall algebra was given. From the viewpoint of associativity, derived Hall algebras generalize Ringel-Hall algebras. The study of the theory of derived Hall algebra is meaningful.  It is applied to characterize refined Donaldson-Thomas invariants via constructing an integration map from the derived Hall algebra of a 3-Calabi-Yau category to a quantum torus~\cite{KS}.  Recently, Hernandez and Leclerc~\cite{HL} defined the monoidal categorification of a derived Hall algebra.  One may hope to construct an analogue of the comultiplication as the case of Ringel-Hall algebras.  In Section 2 and 3 of this note, we define a map over derived Hall algebras analogous to comultiplications of Ringel-Hall algebras. In general, the map does not provide an algebra homomorphism even for the derived category of a hereditary algebra. However, it induces a new multiplication structure on a derived Hall algebra. Then we can write down the Drinfeld dual of a derived Hall algebra. The Drinfeld dual coincides with the finite field version of the motivic Hall algebras introduced by Kontsevich-Soibelman~\cite{KS}. In Section 4, we point out that the method of the paper ~\cite{XX2006} provides two symmetries associated with the octahedral axiom and they are equivalent. the first symmetry implies the associativity of the derived Hall algebra in the sense of To\"{e}n, the second symmetry implies the associativity of its Drinfeld dual.  Section 5 is contributed to construct the motivic version of a derived Hall algebra and show it is associative and isomorphic to the  Kontsevich-Soibelman's motivic Hall algebra.

\section{The Drinfeld dual of an algebra}
Let $A$ be an associative algebra over $\mathbb{Q}$ such that as a space, it has a basis $\{u_{\alpha}\}_{\alpha\in \mathcal{P}}$ and the multiplication is given by
$$
u_{\alpha}\cdot u_{\beta}=\sum_{\lambda\in \mathcal{P}}g_{\alpha\beta}^{\lambda}u_{\lambda},
$$
where $g_{\alpha\beta}^{\lambda}\in \mathbb{Q}$ are the structure constants. We denote by $A\widehat{\otimes} A$ the $\bbq$-space of all formal (possibly infinite) linear combinations $\sum_{\alpha,\beta\in \mathcal{P}}c_{\alpha, \beta}u_{\alpha}\otimes u_{\beta}$ with $c_{\alpha, \beta}\in \bbq$,  which can be viewed as the completion of $A\otimes A.$ Assume that there exists a linear map $\delta:A\rightarrow A\widehat{\otimes} A$ defined by
$$
\delta(u_{\lambda})=\sum_{\alpha,\beta}h_{\lambda}^{\alpha\beta}u_{\alpha}\otimes u_{\beta}.
$$
satisfying that for fixed $\alpha, \beta\in \mathcal{P}$, there are only finitely many $\lambda$ such that $h_{\lambda}^{\alpha\beta}\neq 0.$

Given a non-degenerate symmetric bilinear form $(-, -): A\times A\rightarrow \mathbb{Q}$ such that
$$
(u_{\alpha}, u_{\beta})=\delta_{\alpha, \beta}t_{\alpha}.
$$
for some nonzero $t_{\alpha}\in \bbq.$ Then the bilinear form naturally induces the following two bilinear forms
$$
f_1: (A\otimes A)\times(A\widehat{\otimes} A)\rightarrow \mathbb{Q}
\quad\mbox{ and }\quad
f_2: (A\widehat{\otimes} A)\times(A\otimes A)\rightarrow \mathbb{Q}.
$$
The first bilinear form is defined by $$f_1(\sum_{\alpha, \beta\in \mathcal{P}}c_{\alpha,\beta}u_{\alpha}\otimes u_{\beta}, \sum_{\alpha',\beta'\in \mathcal{P}}d_{\alpha', \beta'}u_{\alpha'}\otimes u_{\beta'})=\sum_{\alpha, \beta\in \mathcal{P}}c_{\alpha,\beta}d_{\alpha,\beta}t_{\alpha}t_{\beta}.$$
Note the sum of the right hand is induced by the first sum of the left hand which is a finite sum. The map $f_2$ is defined similarly.
\begin{Prop}\label{coproduct-Lemma}
For any $a, b, c\in A$, the equality $(a, bc)=(\delta(a), b\otimes c)$ holds if and only if for any $u_{\beta}, u_{\gamma}$ and $u_{\alpha}$, we have $h_{\alpha}^{\beta\gamma}t_{\beta}t_{\gamma}=g_{\beta\gamma}^{\alpha}t_{\alpha}.$
\end{Prop}
\begin{proof}
The proof is the same as \cite[Proposition 7.1]{Ringel1995}. It is enough to consider the case that $a, b$ and $c$ are three basis elements denoted by $u_{\alpha}, u_{\beta}$ and $u_{\gamma}$, respectively. By definition,
$$
(a, bc)=(u_{\alpha}, u_{\beta}u_{\gamma})=(u_{\alpha}, \sum_{\lambda\in \mathcal{P}}g_{\beta\gamma}^{\lambda}u_{\lambda})=g_{\beta\gamma}^{\alpha}t_{\alpha}
$$
and
$$
(\delta(a), b\otimes c)=(\sum_{\xi,\tau\in \mathcal{P}}h_{\alpha}^{\xi\tau}u_{\xi}\otimes u_{\tau}, u_{\beta}\otimes u_{\gamma})=h_{\alpha}^{\beta\gamma}t_{\beta}t_{\gamma}.
$$
The proposition follows.
\end{proof}
Let $A^{Dr}$ be a space over $\bbq$ with the basis $\{v_{\alpha}\}_{\alpha\in \mathcal{P}}$. We define the multiplication by setting
$$
v_{\alpha}*v_{\beta}=\sum_{\lambda\in \mathcal{P}}h_{\lambda}^{\alpha\beta}v_{\lambda}.
$$
\begin{theorem}\label{dual}
Assume $(a, bc)=(\delta(a), b\otimes c)$ for any $a, b, c\in A$. Then there exists an isomorphism
$$
\Phi: A^{Dr}\rightarrow A
$$
by sending $v_{\alpha}$ to $t_{\alpha}^{-1}u_{\alpha}.$
\end{theorem}
\begin{proof}
It is clear that the map $\Phi$ is an isomorphism of vector spaces. For any $\alpha, \beta\in \mathcal{P}$, $\Phi(v_{\alpha}*v_{\beta})=\sum_{\lambda\in \mathcal{P}}h_{\lambda}^{\alpha\beta}t^{-1}_{\lambda}u_{\lambda}.$
Also, we have $\Phi(v_{\alpha})\cdot \Phi(v_{\beta})=t^{-1}_{\alpha}t^{-1}_{\beta}u_{\alpha}\cdot u_{\beta}=\sum_{\alpha, \beta\in \mathcal{P}}t^{-1}_{\alpha}t^{-1}_{\beta}g_{\alpha\beta}^{\lambda}u_{\lambda}.$ Proposition \ref{coproduct-Lemma} concludes that $\Phi$ is an algebra homomorphism, i.e.,
$$
\Phi(v_{\alpha}*v_{\beta})=\Phi(v_{\alpha})\cdot \Phi(v_{\beta}).
$$
\end{proof}
As a corollary, the algebra $A^{Dr}$ is also an associative algebra and called the Drinfeld dual of $A.$ The first canonical example comes from Ringel-Hall algebra.
We recall its definition \cite{Ringel1990},\cite{Schiffmann1}.
\begin{example}
Let $\mathcal{A}$ be a small finitary abelian category and linear over some finite field $k$ with $q$ elements and $\mathcal{P}$ be the set of isomorphism classes of objects in $\mathcal{A}$. For $\alpha\in \mathcal{P}$, we take a representative $V_{\alpha}\in \mathcal{A}$. The Ringel-Hall algebra of $\mathcal{A}$ is a vector space $\mathcal{H}(\mathcal{A})=\mathbb{Q}\mathcal{P}=\oplus_{\alpha\in \mathcal{P}}\mathbb{Q}u_{\alpha}$ with the multiplication
$$
u_{\alpha} \cdot u_{\beta}=\sum_{\lambda\in \mathcal{P}}g_{\alpha\beta}^{\lambda}u_{\lambda}
$$
where $g_{\alpha\beta}^{\lambda}=|\{V\subseteq V_{\lambda}\mid V\cong V_{\alpha}, V_{\lambda}/V\cong V_{\beta}\}|.$
Define the map \cite[Section 1.4]{Schiffmann1}:
$$
\delta: \mathcal{H}(\mathcal{A})\rightarrow \mathcal{H}(\mathcal{A})\widehat{\otimes} \mathcal{H}(\mathcal{A})
$$
satisfying $\delta(u_{\lambda})=\sum_{\alpha, \beta}h_{\lambda}^{\alpha\beta}u_{\alpha}\otimes u_{\beta}$
where $h_{\lambda}^{\alpha\beta}=\frac{|\mathrm{Ext}^1_{\Lambda}(V_{\alpha}, V_{\beta})_{V_{\lambda}}|}{|\mathrm{Hom}_{\Lambda}(V_{\alpha}, V_{\beta})|}$ (see \cite{Schiffmann1}).
For fixed $\alpha, \beta\in \mathcal{P}$,  $\mathrm{dim}_k\mathrm{Ext}^1(V_{\alpha}, V_{\beta})<\infty$. Then there are finitely many $\lambda$ such that $h_{\lambda}^{\alpha\beta}\neq 0.$
The relation between $h_{\lambda}^{\alpha\beta}$ and $g_{\alpha\beta}^{\lambda}$ is given by the Riedtmann-Peng formula
$$
h_{\lambda}^{\alpha\beta}=\frac{|\mathrm{Ext}^1_{\Lambda}(V_{\alpha}, V_{\beta})_{V_{\lambda}}|}{|\mathrm{Hom}_{\Lambda}(V_{\alpha}, V_{\beta})|}=g_{\alpha\beta}^{\lambda}a_{\alpha}a_{\beta}a_{\lambda}^{-1}
$$
where $a_{\alpha}=|\mathrm{Aut}(V_{\alpha})|.$
\newline Define a symmetric bilinear form on $\mathcal{H}(\mathcal{A})$:
$$
(u_{\alpha}, u_{\beta})=\delta_{\alpha\beta}\frac{1}{|\mathrm{Aut}(V_{\alpha})|}=\delta_{\alpha\beta}\frac{1}{a_{\alpha}}
$$
It induces bilinear forms $(\mathcal{H}(\mathcal{A})\otimes \mathcal{H}(\mathcal{A}))\times (\mathcal{H}(\mathcal{A})\widehat{\otimes} \mathcal{H}(\mathcal{A}))\rightarrow \bbq$ and $(\mathcal{H}(\mathcal{A})\widehat{\otimes} \mathcal{H}(\mathcal{A}))\times (\mathcal{H}(\mathcal{A})\otimes \mathcal{H}(\mathcal{A}))\rightarrow \bbq$ by setting
$$
(a_1\otimes a_2, b_1\otimes b_2)=(a_1, b_1)(a_2, b_2).
$$
Using Proposition \ref{coproduct-Lemma} and the Riedtmann-Peng formula, we obtain
$$
 (a, bc)=(\delta(a), b\otimes c).
$$
for any $a, b$ and $c$ in $\mathcal{H}(\mathcal{A}).$
The Drinfeld dual algebra of  $\mathcal{H}(\mathcal{A})$ is a vector space $\mathcal{H}^{Dr}(\mathcal{A})=\oplus_{\alpha\in \mathcal{P}}\mathbb{Q}v_{\alpha}$ with the multiplication
$$
v_{\alpha}* v_{\beta}=\sum_{\lambda}h_{\lambda}^{\alpha\beta}v_{\lambda}.
$$
Theorem \ref{dual} concludes  an isomorphism $\Phi: \mathcal{H}^{Dr}(\mathcal{A})\rightarrow \mathcal{H}(\mathcal{A})
$
by setting $\Phi(v_{\alpha})=a_{\alpha}u_{\alpha}.$
\end{example}
\section{The derived Riedtmann-Peng formula}
We recall some notations and results in \cite{XX2006}. Let $k$ be a
finite field with $q$ elements and $\mc$ a (left) homologically
finite $k$-additive triangulated category with the translation (or
shift) functor $T=[1]$, i.e., a finite $k$-additive triangulated category satisfying the following conditions (see
\cite{XX2006})
\begin{enumerate}
\item[(1)] the homomorphism space $\Hom(X,Y)$ for any two
objects $X$ and $Y$ in $\mc$ is a finite dimensional $k$-space;
\item[(2)] the endomorphism ring $\ed X$ for any indecomposable
object $X$ in $\mc$ is a finite dimensional local $k$-algebra;
\item[(3)] $\mc$ is (left) locally
homological finite, i.e., $\sum_{i\geq 0}\dim_k\Hom(X[i],Y)<\infty$
for any $X$ and $Y$ in $\mc.$
\end{enumerate}
Note that the first two conditions imply the validity of the
Krull--Schmidt theorem in $\mc$, which means that any object in
$\mc$ can be uniquely decomposed into the direct sum of finitely
many indecomposable objects up to isomorphism. For $X\in \mc$, we denote by $[X]$ the isomorphism class of $X.$

For any $X, Y$ and $Z$ in $\mc$, we will use $fg$ to denote the
composition of morphisms $f:X\rightarrow Y$ and $g:Y\rightarrow Z,$
and $|S|$ to denote the cardinality of a finite set $S.$

Given $X,Y;L\in \mc,$ put
$$ W(X,Y;L)=\{(f,g,h)\in \Hom(X,L)\times
\Hom(L,Y)\times \Hom(Y,X[1])\mid$$$$
X\xrightarrow{f}L\xrightarrow{g}Y\xrightarrow{h}X[1] \mbox{ is a
triangle}\}.$$ \\
There is a natural action of $\aut X\times \aut Y$ on $W(X,Y;L).$
The orbit of $(f,g,h)\in W(X,Y;L)$ is denoted by
$$(f,g,h)^{\wedge}:=\{(af,gc^{-1},ch(a[1])^{-1})\mid (a,c)\in \aut X\times \aut Y\}.$$
The orbit space is denoted by $V(X,Y;L)=\{ (f,g,h)^{\wedge}|(f,g,h)
\in W(X,Y;L)\} .$ The radical of $\Hom(X, Y)$ is denoted by
$\mathrm{radHom}(X, Y)$ which is the set $$\{f\in \Hom(X, Y)\mid gfh
\mbox{ is not an isomorphism }\mbox{for any }g: A\rightarrow X\mbox{
and }
$$
$$h:Y\rightarrow A\mbox{ with } A\in \mc\mbox{ indecomposable }\}. $$
For any $L\xrightarrow{n}Z[1],$  there exist the decompositions
$L=L_{1}(n)\oplus L_{2}(n)$, $Z[1]=Z_{1}[1](n)\oplus Z_2[1](n)$ and
$b\in \aut L,$ $d\in \aut Z$ such that $bn(d[1])^{-1}=\left(
         \begin{array}{cc}
           n_{11} & 0 \\
           0 & n_{22} \\
         \end{array}
       \right)$ and the induced maps
$n_{11}:L_{1}(n)\rightarrow Z_{1}[1](n)$ is an isomorphism and
$n_{22}:L_{2}(n)\rightarrow Z_2[1](n)$ belongs to
$\mathrm{radHom}(L_{2}(n), Z_2[1](n))$. The above decomposition only
depends on the equivalence class  of $n$ up to an isomorphism. Let
$\alpha=(l,m,n)^{\wedge}\in V(Z,L;M),$ the classes of $\alpha$ and
$n$ are determined to each other in $V(Z,L;M).$ We may  denote $n$
by $n(\alpha)$ and $L_1(n)$ by $L_1(\alpha)$ respectively.

Denote by $\Hom(X,Y)_Z$ the subset of $\Hom(X,Y)$ consisting of the
morphisms whose mapping cones are isomorphic to $Z.$ For
$X,Y\in\mc,$ we set
$$\{X,Y\}:=\prod_{i>0}|\mathrm{Hom}(X[i],Y)|^{(-1)^{i}}.$$ By
checking the stable subgroups of automorphism groups, we have the
following proposition.

\begin{Prop}\cite[Proposition
2.5]{XX2006}\label{mainproposition} For any $M, L$ and $Z$ in $\mc$,
we have the following equalities
$$
\frac{|\mathrm{Hom}(M,L)_{Z[1]}|}{|\mathrm{Aut}L|}\cdot
\frac{\{M,L\}}{\{Z,L\}\cdot\{L,L\}}=\sum_{\alpha\in
V(Z,L;M)}\frac{|\mathrm{End}L_1(\alpha)|}{|\mathrm{Aut}L_1(\alpha)|},
$$
$$
\frac{|\mathrm{Hom}(Z,M)_{L}|}{|\mathrm{Aut}Z|}\cdot
\frac{\{Z,M\}}{\{Z,L\}\cdot\{Z,Z\}}=\sum_{\alpha\in
V(Z,L;M)}\frac{|\mathrm{End}L_1(\alpha)|}{|\mathrm{Aut}L_1(\alpha)|}
$$
\end{Prop}

Using this proposition, one can easily  deduce the following
corollary.
\begin{Cor}\label{RP}
For any $X, Y$ and $L$ in $\mc$, we have
$$
\frac{|\mathrm{Hom}(Y,X[1])_{L[1]}|}{|\mathrm{Aut}X|}\cdot
\frac{\{Y,X[1]\}}{\{X,X\}}=
\frac{|\mathrm{Hom}(L,Y)_{X[1]}|}{|\mathrm{Aut}L|}\cdot
\frac{\{L,Y\}}{\{L,L\}}
$$
and
$$
\frac{|\mathrm{Hom}(Y[-1],X)_{L}|}{|\mathrm{Aut}Y|}\cdot
\frac{\{Y[-1],X\}}{\{Y,Y\}}=
\frac{|\mathrm{Hom}(X,L)_{Y}|}{|\mathrm{Aut}L|}\cdot
\frac{\{X,L\}}{\{L,L\}}.
$$
\end{Cor}

Let $\mathcal{A}$ be a finitary abelian category and $X, Y, L\in \mathcal{A}.$ Define $$ E(X,Y;L)=\{(f,g)\in \Hom(X,L)\times \Hom(L,Y)\mid
$$ $$ 0 \rightarrow X \xrightarrow{f} L \xrightarrow{g}Y \rightarrow
0 \mbox{ is an exact sequence}\}.$$ The group $\aut X\times \aut Y$
acts freely on $E(X,Y;L)$ and the orbit of $(f,g)\in E(X,Y;L)$ is
denoted by $(f,g)^{\wedge}:=\{(af,gc^{-1})\mid (a,c)\in \aut X\times
\aut Y\}.$ If the orbit space is denoted by $O(X,Y;L)=\{
(f,g)^{\wedge}|(f,g) \in E(X,Y;L)\}$, then the \emph{Hall number}
$\mathrm{g}_{XY}^L=|O(X,Y;L)|$. It is easy to see

$$\mathrm{g}_{XY}^L=\displaystyle\frac{|\mathcal{M}(X,L)_{Y}|}{|\aut X|}=\displaystyle\frac{|\mathcal{M}(L,Y)_{X}|}{|\aut Y|},$$ where $\mathcal{M}(X,L)_{Y}$ is the subset of
$\textrm{Hom}(X,L)$ consisting of monomorphisms $f: X
\hookrightarrow L$ whose cokernels $\textrm{Coker}(f)$ are
isomorphic to $Y$ and $\mathcal{M}(L,Y)_{X}$ is dually defined.

The equality in Corollary \ref{RP} can be considered as a generalization of  the Riedtmann-Peng formula in
abelian categories to  homologically finite triangulated categories.
Indeed, assume that $\mc=\md^b(\mathcal{A})$ for a finitary abelian
category $\mathcal{A}$ and $X, Y$ and $L\in \mathcal{A}.$
Then one
can obtain
$$
\mathrm{Hom}(Y,X[1])_{L[1]}=\mathrm{Ext}^1(Y, X)_L, \quad
\{Y,X[1]\}=|\mathrm{Hom}_{\mathcal{A}}(Y, X)|^{-1},
$$
where $\mathrm{Ext}^1(X, Y)_L$ is the set of equivalence class of
extension of Y by X with the middle term isomorphic to $L$ and
$$
g_{XY}^L=\frac{|\mathrm{Hom}(L,Y)_{X[1]}|}{|\mathrm{Aut}Y|},\quad
\{X,X\}=\{L,L\}=\{L,Y\}=0.
$$
Under the assumption, Corollary \ref{RP} is reduced to the
Riedtmann-Peng formula (\cite{Rie1994}\cite{Peng}).
$$
\frac{|\mathrm{Ext}^1(Y, X)_L|}{|\mathrm{Hom}_{\mathcal{A}}(Y,
X)|}=g_{XY}^L\cdot |\mathrm{Aut}X|\cdot
|\mathrm{Aut}Y|\cdot|\mathrm{Aut}L|^{-1}.
$$
For any $X, Y$ and $L\in \mc$, set
$$
F_{XY}^L=\frac{|\mathrm{Hom}(L,Y)_{X[1]}|}{|\mathrm{Aut}Y|}\cdot
\frac{\{L,Y\}}{\{Y,Y\}}=\frac{|\mathrm{Hom}(X,L)_{Y}|}{|\mathrm{Aut}X|}\cdot
\frac{\{X,L\}}{\{X,X\}}.
$$

\begin{theorem}(\cite{Toen2005},\cite{XX2006})\label{maintheorem123}
Let $\mathcal{H}(\mc)$ be the vector space over $\bbq$ with the
basis $\{u_{[X]}\mid X\in \mc\}$. Endowed with the multiplication
defined by
$$
u_{[X]}\cdot u_{[Y]}=\sum_{[L]}F_{XY}^L u_{[L]},$$ $\mathcal{H}(\mc)$ is
an associative algebra with the unit  $u_{[0]}$.
\end{theorem}
The algebra $\mathcal{H}(\mc)$ is called the derived Hall algebra when $\mc$ is a derived category. Here, we also use this name for a general left homologically finite triangulated category.

Now we define the Drinfeld dual of $\mathcal{H}(\mc)$. Set $$h_{L}^{XY}=|\mathrm{Hom}_{\mc}(Y, X[1])_{L}|\cdot \{Y, X[1]\}.$$
Define a map $\delta: \mathcal{H}(\mc)\rightarrow \mathcal{H}(\mc)\widehat{\otimes}\mathcal{H}(\mc)$ by
$$
\delta(u_{[L]})=\sum_{[X], [Y]}h_{L}^{XY}u_{[X]}\otimes u_{[Y]}.
$$
Define a symmetric bilinear form:
$$
(u_{[X]}, u_{[Y]})=\delta_{[X], [Y]}\frac{1}{|\mathrm{Aut}X|\{X, X\}}
$$
It induces bilinear forms
$$
(\mathcal{H}(\mc)\widehat{\otimes}\mathcal{H}(\mc))\times (\mathcal{H}(\mc)\otimes\mathcal{H}(\mc))\rightarrow \mathbb{Q} \mbox{  and }(\mathcal{H}(\mc)\otimes\mathcal{H}(\mc))\times (\mathcal{H}(\mc)\widehat{\otimes}\mathcal{H}(\mc))\rightarrow \mathbb{Q}
$$  by setting
$$
(a_1\otimes a_2, b_1\otimes b_2)=(a_1, b_1)(a_2, b_2).
$$
Set $t_{[X]}=\frac{1}{|\mathrm{Aut}X|\{X, X\}}$. Then the derive Riedtmann-Peng formula in Corollary \ref{RP} can be written as
$$
h_{L}^{XY}t_{[X]}t_{[Y]}=F_{XY}^{L}t_{[L]}
$$
for any $X, Y$ and $L$ in $\mc$.
Using Proposition \ref{coproduct-Lemma}, we have
$$ (a, bc)=(\delta(a), b\otimes c),\quad\forall a, b, c\in \mathcal{H}(\mc).$$ The Drinfeld dual algebra is a $\bbq$-space $\mathcal{H}^{Dr}(\mc)$ with the basis $\{v_{[X]}\mid X\in \mc\}$ and the multiplication
\begin{eqnarray}
  v_{[X]}*v_{[Y]}=\sum_{[L]}h_{L}^{XY}v_{[L]} &=& \{Y,X[1]\}\cdot
\sum_{[L]}|\mathrm{Hom}(Y,X[1])_{L[1]}| v_{[L]}\nonumber \\
   &=& \{Y[-1],X\}\cdot
\sum_{[L]}|\mathrm{Hom}(Y[-1],X)_{L}| v_{[L]}.\nonumber
\end{eqnarray}
In \cite{KS}, Kontsevich and Soibelman defined the motivic Hall
algebra for an ind-constructible triangulated $A_{\infty}$-category.
We can define the finite field version of a motivic Hall algebra for a homologically finite $k$-additive
triangulated category, which is just $\mathcal{H}^{Dr}(\mc)$.
Following Theorem \ref{dual}, we have the immediate result.
\begin{Cor}\cite[Proposition 6.12]{KS}
The map $\Phi:
\mathcal{H}^{Dr}(\mathcal{C})\rightarrow
\mathcal{H}(\mathcal{C})$ by $\Phi(v_{[X]})=|\mathrm{Aut}X|\cdot\{X,
X\}\cdot u_{[X]}$ for any $X\in \mathcal{C}$ is an algebraic isomorphism between
$\mathcal{H}^{Dr}(\mathcal{C})$ and
$\mathcal{H}(\mathcal{C})$.
\end{Cor}
Then Theorem \ref{maintheorem123} implies that the algebra $\mathcal{H}^{Dr}(\mathcal{C})$ is an associative algebra.
 \bigskip

In order to introduce the extended twisted derived Hall algebra $\mathcal{H}_{et}(\mc)$ of $\mathcal{H}(\mc),$ we need  more assumptions. Assume that $\mc$ is homologically finite, i.e.,
$$\sum_{i\in \bbz}\dim_k\Hom(X[i],Y)<\infty$$
for any $X$ and $Y$ in $\mc.$  For example, the derived category $\mc=\md^b(\mathcal{A})$ for a small finitary abelian category $\mathcal{A}$ is homologically finite. Let $K(\mc)$ be the Grothendieck group of $\mc$, one can define a bilinear form $\langle-, -\rangle: K(\mc)\times K(\mc)\rightarrow \bbz$ by setting
$$
\langle[X], [Y]\rangle=\sum_{i\in\bbz}(-1)^i\mathrm{dim}_{k}\mathrm{Hom}_{\mc}(X, Y[i])
$$
for $X, Y\in \mc.$ It induces the symmetric bilinear form $(-, -)$ on $K(\mc)$ by defining $([X], [Y])=\langle[X], [Y]\rangle+\langle[Y], [X]\rangle.$ For convenience, for any
object $X\in \mc,$  we use the same notation $[X]$ for the isomorphism class and its image in $K(\mc)$. It is easy to check that this bilinear form is well-defined.
Set $v=\sqrt{q}$. Then $\mathcal{H}_{et}(\mc)$ is given by the $\bbq(v)$-space with basis $\{K_{\alpha}u_{[X]}\mid \alpha\in K(\mc), X\in \mc\}$ and the multiplication defined by
$$
(K_{\alpha}u_{[X]})\circ (K_{\beta}u_{[Y]})=v^{\langle[X], [Y]\rangle-(\beta, [X])}K_{\alpha+\beta}u_{[X]}\cdot u_{[Y]}.
$$
Note that $K_0=u_0=1.$
\begin{Prop}
The algebra $H_{et}(\mc)$ is associative.
\end{Prop}
\begin{proof}By definition, we have
\begin{eqnarray}
   && [(K_{\alpha}u_{[X]})\circ (K_{\beta}u_{[Y]})]\circ(K_{\gamma}u_{[Z]})\nonumber \\
   &=& [v^{\langle[X], [Y]\rangle-(\beta,[X])}K_{\alpha+\beta}u_{[X]}\cdot u_{[Y]}]\circ (K_{\gamma}u_{[Z]})\nonumber \\
   &=& v^{\langle[X], [Y]\rangle+\langle[L], [Z]\rangle-(\beta,[X])-(\gamma, [L])}K_{\alpha+\beta+\gamma}\sum_{[L]}F_{XY}^L u_{[L]}\cdot u_{[Z]} \nonumber\\
   &=& v^{\langle[X], [Y]\rangle+\langle[X]+[Y], [Z]\rangle-(\beta,[X])-(\gamma, [X]+[Y])}K_{\alpha+\beta+\gamma}(u_{[X]}\cdot u_{[Y]})\cdot u_{[Z]}\nonumber\\
   &=& v^{\langle[X], [Y]+[Z]\rangle+\langle[Y], [Z]\rangle-(\beta+\gamma,[X])-(\gamma, [Y])}K_{\alpha+\beta+\gamma}u_{[X]}\cdot (u_{[Y]}\cdot u_{[Z]})\nonumber\\
   &=& (K_{\alpha}u_{[X]})\circ [(K_{\beta}u_{[Y]})]\circ(K_{\gamma}u_{[Z]})].\nonumber
\end{eqnarray}
\end{proof}
Dually, one can define $\mathcal{H}_{et}^{-}(\mc)$ with the basis $\{K_{\alpha}u^{-}_{[X]}\mid \alpha\in K(\mc), X\in \mc\}$ and the multiplication
$$
(K_{\alpha}u^{-}_{[X]})\circ (K_{\beta}u^{-}_{[Y]})=v^{\langle[X], [Y]\rangle+(\beta, [X])}K_{\alpha+\beta}u^{-}_{[X]}\cdot u^{-}_{[Y]},
$$
where
$$
u^{-}_{[X]}\cdot u^{-}_{[Y]}=\sum_{[L]}F_{XY}^{L}u^{-}_{[L]}.
$$
There exist a map $\delta: \mathcal{H}_{et}(\mc)\rightarrow \mathcal{H}_{et}(\mc)\widehat{\otimes}\mathcal{H}_{et}(\mc)$ by setting
$$
\delta(K_{\gamma}u_{[L]})=\sum_{[X], [Y]}v^{\langle[X], [Y]\rangle}h_{L}^{XY}K_{\gamma}u_{[X]}K_{[Y]}\otimes K_{\gamma}u_{[Y]}
$$
and  a bilinear form on $\mathcal{H}_{et}(\mc)\times \mathcal{H}^{-}_{et}(\mc)$ defined by
$$
(K_{\alpha}u_{[X]}, K_{\beta}u^{-}_{[Y]})=v^{-(\alpha, \beta)-(\beta, [X])+(\alpha, [Y])}\delta_{[X], [Y]}\frac{1}{|\mathrm{Aut}X|\{X, X\}}.
$$
It naturally induces a bilinear form $$(\mathcal{H}_{et}(\mc)\widehat{\otimes}\mathcal{H}_{et}(\mc))\times (\mathcal{H}^{-}_{et}(\mc)\otimes\mathcal{H}^{-}_{et}(\mc))\rightarrow \mathbb{Q}(v).$$
\begin{Prop}
For any $a\in \mathcal{H}_{et}(\mc)$ and $b, c\in \mathcal{H}_{et}^{-}(\mc)$, we have $$(a, bc)=(\delta(a), b\otimes c).$$
\end{Prop}
\begin{proof}
It is enough to consider the case when $a=K_{\gamma}u_{[L]}, b=K_{\alpha}u^{-}_{[X]}$ and $c=K_{\beta}u^{-}_{[Y]}.$ By definition,
$$
(K_{\gamma}u_{[L]}, (K_{\alpha}u^{-}_{[X]})\circ (K_{\beta}u^{-}_{[Y]}))$$$$=v^{\langle[X], [Y]\rangle+(\beta, [X])}(K_{\gamma}u_{[L]}, K_{\alpha+\beta}u^{-}_{[X]}\cdot u^{-}_{[Y]})
$$
$$
\hspace{-0.2cm}=v^{\langle[X], [Y]\rangle+(\beta, [X])}F_{XY}^{L}(K_{\gamma}u_{[L]}, K_{\alpha+\beta}u^{-}_{[L]})
$$
$$
\hspace{1.3cm}=v^{\langle[X], [Y]\rangle+(\beta, [X])-(\gamma, \alpha+\beta)-(\alpha+\beta, L)+(\gamma, [L])}F_{XY}^L\cdot t_{[L]}.
$$
$$
\hspace{1.9cm}=v^{\langle[X], [Y]\rangle-(\gamma, \alpha+\beta-[X]-[Y])-(\alpha, [X]+[Y])-(\beta, [Y])}F_{XY}^L\cdot t_{[L]}.
$$
Similarly, by definition, we have
$$
\hspace{-6cm}(\delta(K_{\gamma}u_{[L]}), (K_{\alpha}u^{-}_{[X]})\otimes (K_{\beta}u^{-}_{[Y]}))$$$$\hspace{-2cm}=h_{L}^{XY}v^{\langle[X], [Y]\rangle}((K_{\gamma}u_{[X]}K_{[Y]})\otimes (K_{\gamma}u_{[Y]}), (K_{\alpha}u^{-}_{[X]})\otimes (K_{\beta}u^{-}_{[Y]}))$$$$\hspace{-2cm}=h_{L}^{XY}v^{\langle[X], [Y]\rangle-([X], [Y])}(K_{\gamma+[Y]}u_{[X]}, K_{\alpha}u^{-}_{[X]})\cdot (K_{\gamma}u_{[Y]}, K_{\beta}u^{-}_{[Y]})$$
$$
\hspace{0.3cm}=h_{L}^{XY}v^{\langle[X], [Y]\rangle-([X], [Y])-(\gamma+[Y], \alpha)-(\alpha, [X])+(\gamma+[Y], [X])-(\gamma, \beta)-(\beta, [Y])+(\gamma, [Y])}t_{[X]}t_{[Y]}.
$$
$$
\hspace{-3.6cm}=v^{\langle[X], [Y]\rangle+(\beta, [X])-(\gamma, \alpha+\beta)-(\alpha+\beta, L)+(\gamma, [L])}h_{L}^{XY}t_{[X]}t_{[Y]}.
$$
This completes the proof.
\end{proof}
\begin{Prop}
Let $\mathcal{H}^{Dr}_{et}$ be the $\bbq(v)$-space with the basis $\{K_{\alpha}\theta_{[X]}\mid \alpha\in K(\mc), X\in \mc\}$ and
the multiplication given by
$$
K_{\alpha}\theta_{[X]}*K_{\beta}\theta_{[Y]}=v^{(\beta, [X])-2(\alpha, \beta)}\sum_{[L]}v^{\langle[X], [Y]\rangle}h_{L}^{XY}K_{\alpha+\beta}\theta_{[L]}.
$$
Then the map $\Phi: \mathcal{H}^{Dr}_{et}\rightarrow \mathcal{H}^{-}_{et}$ defined by $\Phi(K_{\alpha}\theta_{[X]})=v^{(\alpha, \alpha)}t^{-1}_{[X]}K_{\alpha}u^{-}_{[X]}$ is an isomorphism of algebras.
\end{Prop}
By this proposition, we can view $\mathcal{H}^{Dr}_{et}$ as the Drinfeld double of $\mathcal{H}^{-}_{et}.$

\section{Two symmetries}
Consider the following commutative diagram in $\mc,$ which is a pushout and a pullback in the mean time,
$$
\xymatrix{L'\ar[d]^{m'}\ar[r]^{f'}&M\ar[d]^{m}\\X\ar[r]^{f}&L}
$$
Applying the Octahedral axiom, one can obtain the following commutative
diagram:
\begin{equation}\label{maindiagramme}
\xymatrix{Z\ar@{=}[r]\ar@{.>}[d]^{l'} &Z\ar[d]^l\\
L'\ar[r]^{f'}\ar[d]^{m'} &M\ar@{.>}[r]^{g'}\ar[d]^m &Y\ar@{.>}[r]^-{h'}\ar@{=}[d] &L'[1]\ar[d]^{m'[1]}\\
X\ar[r]^f\ar@{.>}[d]^{n'} &L\ar[r]^g\ar[d]^n &Y\ar[r]^-{h} &X[1]\\
Z[1]\ar@{=}[r] &Z[1]}
\end{equation}
with rows and columns being distinguished triangles and a distinguished triangle
\begin{equation}\label{maintriangle}
\xymatrix{L'\ar[rr]^-{(%
\begin{array}{cc}
  f' & -m' \\
\end{array}%
)}&&M\oplus X\ar[rr]^-{\left(%
\begin{array}{c}
  m \\
  f \\
\end{array}%
\right)}&&L\ar[rr]^{\begin{array}{c}
  \theta \\
\end{array}}&&L'[1].}
\end{equation}

The above triangle induces two sets
$$
\mathrm{Hom}(M\oplus X,L)^{Y,Z[1]}_{L'[1]}:=\{ \left(\begin{array}{c}
  m \\
  f \\
\end{array}\right) \in
\mathrm{Hom}(M\oplus X,L)\mid$$$$\hspace{1cm} Cone(f)\simeq Y,
Cone(m)\simeq Z[1] \mbox{ and }Cone \left(\begin{array}{c}
  m \\
  f \\
\end{array}\right)\simeq L'[1]\}
$$
and
$$
\mathrm{Hom}(L',M\oplus X)^{Y,Z[1]}_{L}:=\{(f',-m')\in
\mathrm{Hom}(L',M\oplus X)\mid $$$$\hspace{1cm}Cone(f')\simeq Y,
Cone(m')\simeq Z[1] \mbox{ and } Cone(f',-m')\simeq L\}.
$$
\\
\noindent\textbf{The symmetry-I}: The orbit space of
$\mathrm{Hom}(M\oplus X,L)^{Y,Z[1]}_{L'[1]}$ under the action of
$\mathrm{Aut} L$ and the orbit space of
$\mathrm{Hom}(L',M\oplus X)^{Y,Z[1]}_{L}$ under the action of
$\mathrm{Aut} L'$ coincide. More explicitly,
the symmetry implies the identity:
$$
\frac{|\mathrm{Hom}(M\oplus X,L)^{Y,Z[1]}_{L'[1]}|}{|\mathrm{Aut}
L|}\frac{\{M\oplus
X,L\}}{\{L',L\}\{L,L\}}=\frac{|\mathrm{Hom}(L',M\oplus
X)^{Y,Z[1]}_{L}|}{|\mathrm{Aut} L'|}\frac{\{L',M\oplus
X\}}{\{L',L\}\{L',L'\}}.
$$
\begin{proof}
The equality is a direct application of Proposition
\ref{mainproposition} to the Triangle \ref{maintriangle}.
\end{proof}

Roughly speaking, \textbf{the symmetry-I} compares
 $$\xymatrix{L'\ar[rr]^-{(%
\begin{array}{cc}
  f' & -m' \\
\end{array}%
)}&&M\oplus X&\mbox{and}&M\oplus X\ar[rr]^-{\left(%
\begin{array}{c}
  m \\
  f \\
\end{array}%
\right)}&&L}$$ in the Triangle \ref{maintriangle}.

The diagram
\ref{maindiagramme} induces \textbf{a new symmetry} which compares
 $$\xymatrix@C=0.5cm{
   L' \ar[r]^{f'}& M\ar[r]^{m}&L   & \mbox{and} &
   L' \ar[r]^{m'} & X\ar[r]^{f}&L }.$$
Using the derived Riedtmann-Peng formula (Corollary \ref{RP}), we
have
$$
\frac{|\mathrm{Hom}(Y[-1], L')_{M}|}{|\mathrm{Aut}Y|}\cdot
\frac{\{Y[-1],L'\}}{\{Y,Y\}}=
\frac{|\mathrm{Hom}(L',M)_{Y}|}{|\mathrm{Aut}M|}\cdot
\frac{\{L',M\}}{\{M,M\}},
$$
$$
\frac{|\mathrm{Hom}(L,Z[1])_{M[1]}|}{|\mathrm{Aut}Z|}\cdot
\frac{\{L,Z[1]\}}{\{Z,Z\}}=
\frac{|\mathrm{Hom}(M,L)_{Z[1]}|}{|\mathrm{Aut}M|}\cdot
\frac{\{M,L\}}{\{M,M\}},
$$
$$
\frac{|\mathrm{Hom}(X, Z[1])_{L'[1]}|}{|\mathrm{Aut}Z|}\cdot
\frac{\{X,Z[1]\}}{\{Z,Z\}}=
\frac{|\mathrm{Hom}(L',X)_{Z[1]}|}{|\mathrm{Aut}L'|}\cdot
\frac{\{L',X\}}{\{L',L'\}},
$$
and
$$
\frac{|\mathrm{Hom}(Y[-1],X)_{L}|}{|\mathrm{Aut}Y|}\cdot
\frac{\{Y[-1],X\}}{\{Y,Y\}}=
\frac{|\mathrm{Hom}(X,L)_{Y}|}{|\mathrm{Aut}L|}\cdot
\frac{\{X,L\}}{\{L,L\}}.
$$

Hence, one can convert to compare
 $$\xymatrix@C=0.5cm{Y\ar[r]^-{h'}&L'[1], L\ar[r]^-{n}&Z[1]& \mbox{and} & X\ar[r]^-{n'}&Z[1], Y\ar[r]^-{h}&X[1]}$$
in the Diagram \ref{maindiagramme}. In order to describe the second
symmetry, we need to introduce some notations.
Fix $X, Y, Z, M, L$ and $L'$, define
$$
\mathcal{D}_{L, L'}=\{(m, f, h, n)\in \mathrm{Hom}(M, L)\times \mathrm{Hom}(X, L)\times \mathrm{Hom}(Y, X[1])\times \mathrm{Hom}(L, Z[1])\mid
$$
$$
(m, f, h, n) \mbox{ induces a diagram with the form as Diagram \ref{maindiagramme}}\}
$$
and
$$
\mathcal{D}_{L', L}=\{(f', m', h', n')\in \mathrm{Hom}(L', X)\times \mathrm{Hom}(L', M)\times \mathrm{Hom}(Y, L'[1])\times \mathrm{Hom}(X, Z[1])\mid
$$
$$
(m', f', h', n') \mbox{ induces a diagram with the form as Diagram \ref{maindiagramme}}\}.
$$
Here, we say ``$(m, f, h, n)$ induces a diagram with the form as Diagram \ref{maindiagramme}'' means that  there exist morphisms $m', f', h', n', g, g', l$ and $l'$ such that all morphisms constitute a diagram formed as Diagram \ref{maindiagramme}. The crucial point is that the following diagram
$$
\xymatrix{L'\ar[d]^{m'}\ar[r]^{f'}&M\ar[d]^{m}\\X\ar[r]^{f}&L}
$$
is both a pushout and a pullback and rows and columns in Diagram \ref{maindiagramme} are distinguished triangles. Note that the pair $(f',m')$ is uniquely determined by $(m, f, h, n)$ up to isomorphisms under requirement, so the above notation is well-defined.

 There exist natural projections:
$$
p: \mathcal{D}_{L, L'}\longrightarrow \mathrm{Hom}(Y, X[1])\times\mathrm{Hom}(L, Z[1]),
$$
$$
i_1: \mathrm{Hom}(Y, X[1])\times\mathrm{Hom}(L, Z[1])\longrightarrow \mathrm{Hom}(Y, X[1])$$ and $$i_2: \mathrm{Hom}(Y, X[1])\times\mathrm{Hom}(L, Z[1])\longrightarrow \mathrm{Hom}(L, Z[1]).
$$
The image of $i_1\circ p$ is denoted by $\mathrm{Hom}(Y,
X[1])_{L[1]}^{L'}$ and given $h\in \mathrm{Hom}(Y, X[1])_{L[1]}^{L'}$, define $\mathrm{Hom}(L, Z[1])_{M[1]}^{h,
L'}$ to be $i_2\circ p^{-1}\circ i_1^{-1}(h).$ It is clear that $$\mathrm{Hom}(Y,
X[1])_{L[1]}=\bigsqcup_{[L']}\mathrm{Hom}(Y, X[1])_{L[1]}^{L'}.$$
Similarly, there exist projections:
$$
q: \mathcal{D}_{L', L}\rightarrow \mathrm{Hom}(Y, L'[1])\times\mathrm{Hom}(X, Z[1]),
$$
$$
j_1: \mathrm{Hom}(Y, L'[1])\times\mathrm{Hom}(X, Z[1])\rightarrow \mathrm{Hom}(Y, L'[1])$$ and $$j_2: \mathrm{Hom}(Y, L'[1])\times\mathrm{Hom}(X, Z[1])\rightarrow \mathrm{Hom}(X, Z[1]).
$$
The image of $j_1\circ q$ is denoted by $\mathrm{Hom}(X, Z[1])_{L'[1]}^L$
and for any $n'\in \mathrm{Hom}(X, Z[1])_{L'[1]}$, we denote $j_2\circ p^{-1}\circ j_1^{-1}(n')$ by $\mathrm{Hom}(Y, L'[1])_{M[1]}^{n', L}$.
\\
\noindent\textbf{The symmetry-II}:\\
\begin{itemize}
  \item Fix $h\in \mathrm{Hom}(Y, X[1])_{L[1]}^{L'}$, then there exists a surjective map
$$f_{*}: \mathrm{Hom}(L, Z[1])_{M[1]}^{h, L'}\rightarrow
\mathrm{Hom}(X, Z[1])_{L'[1]}^{L}$$ such that the cardinality of any fibre is
$$
|(f^{*})^{-1}|:=|\mathrm{Hom}(Y, Z[1])|\cdot \{X\oplus Y, Z[1]\}\cdot \{L,
Z[1]\}^{-1};
$$
  \item Fix $n'\in \mathrm{Hom}(X, Z[1])_{L'[1]}^{L}$, then there exists a surjective map
$$(m')_{*}: \mathrm{Hom}(Y, L'[1])_{M[1]}^{n',L}\rightarrow
\mathrm{Hom}(Y, X[1])_{L[1]}^{L'}$$ such that the cardinality of any fibre is
$$
|(m')^{-1}|:=|\mathrm{Hom}(Y, Z[1])|\cdot \{Y, X[1]\oplus Z[1]\}\cdot\{Y,
L'[1]\}^{-1};
$$
  \item $|(f_*)^{-1}|\cdot\{Y, X[1]\}\cdot\{L, Z[1]\}=|(m')_*^{-1}|\cdot\{X, Z[1]\}\cdot\{Y, L'[1]\}.$
\end{itemize}
\begin{proof}
Given $h\in \mathrm{Hom}(Y, X[1])_{L[1]}^{L'},$ there exists a triangle$$
\alpha: X\xrightarrow{f} L\xrightarrow{g} Y\xrightarrow{h} X[1].
$$ Applying the functor
$\mathrm{Hom}(-, Z[1])$ on the triangle, one can obtain a long exact sequence
$$
\xymatrix{\cdots\ar[r]&\mathrm{Hom}(Y, Z[1])\ar[r]^{u}&\mathrm{Hom}(L, Z[1])\ar[r]^{v}&\mathrm{Hom}(X, Z[1])\ar[r]&\cdots}.
$$
Then the cardinality of $\mathrm{Im}(u)$ is $|\mathrm{Hom}(Y, Z[1])|\cdot \{X\oplus Y, Z[1]\}\cdot \{L,
Z[1]\}^{-1}.$ The map $f_*$ is the restriction of $v$ to $\mathrm{Hom}(L, Z[1])_{M[1]}^{h, L'}$. By definition, $f_*$ is epic and the fibre is isomorphic to $\mathrm{Ker}(v)=\mathrm{Im}(u).$ Thus the first statement is obtained.
The second statement can be proved in the same way. The third statement
is a direct confirmation.
\end{proof}

Note that there also exist projections:
$$
p_{12}: \mathcal{D}_{L, L'}\rightarrow \mathrm{Hom}(M\oplus X,L)
$$
with the image $\mathrm{Hom}(M\oplus X,L)^{Y,Z[1]}_{L'[1]}$
and
$$
q_{12}: \mathcal{D}_{L', L}\rightarrow \mathrm{Hom} (L', M\oplus X)
$$
with the image $\mathrm{Hom}(L',M\oplus X)^{Y,Z[1]}_{L}.$ \textbf{The symmetry-I} characterizes the relation between $\mathrm{Im}p_{12}$ and $\mathrm{Im}q_{12}$.
Meanwhile, \textbf{the symmetry-II} characterizes the relation between $\mathrm{Im}p$ and $\mathrm{Im}q$. The relation between $\mathrm{Im}p_{12}$ and $\mathrm{Im}p$
($\mathrm{Im}q_{12}$ and $\mathrm{Im}q$) is implicitly shown by the derived Riedtmann-Peng formula (Corollary \ref{RP}). More explicitly, consider the projections
$$
t_1: \mathrm{Hom}(M\oplus X, L)\rightarrow \mathrm{Hom}(X, L);$$ $$t_2: \mathrm{Hom}(M\oplus X, L)\rightarrow \mathrm{Hom}(M, L);
$$
$$
s_1: \mathrm{Hom} (L', M\oplus X)\rightarrow \mathrm{Hom}(L', X);$$  and $$s_2: \mathrm{Hom}(L', M\oplus X)\rightarrow \mathrm{Hom}(L', M).
$$
Using Corollary \ref{RP}, we obtain
$$
\frac{|\mathrm{Hom}(Y,X[1])_{L[1]}^{L'}|}{|\mathrm{Aut}Y|}\cdot
\frac{\{Y,X[1]\}}{\{Y,Y\}}=
\frac{|\mathrm{Im} t_1\circ p_{12}|}{|\mathrm{Aut}L|}\cdot
\frac{\{X,L\}}{\{L,L\}}.
$$
Given a triangle $
\alpha: X\xrightarrow{f} L\xrightarrow{g} Y\xrightarrow{h} X[1]$ with $h\in \mathrm{Hom}(Y,X[1])_{L[1]}^{L'}$, applying Corollary \ref{RP} again, we have
$$
\frac{|\mathrm{Hom}(L,Z[1])_{M[1]}^{h, L'}|}{|\mathrm{Aut}Z|}\cdot
\frac{\{L,Z[1]\}}{\{Z,Z\}}=
\frac{|t_2\circ p_{12}^{-1}\circ t_1^{-1}(f)|}{|\mathrm{Aut}M|}\cdot
\frac{\{M,L\}}{\{M,M\}}.
$$
In the same way, we have
$$
\frac{|\mathrm{Hom}(X, Z[1])_{L'[1]}^{L}|}{|\mathrm{Aut}Z|}\cdot
\frac{\{X,Z[1]\}}{\{Z,Z\}}=
\frac{|\mathrm{Im} s_1\circ q_{12}|}{|\mathrm{Aut}L'|}\cdot
\frac{\{L',X\}}{\{L',L'\}},
$$
and
$$
\frac{|\mathrm{Hom}(Y, L'[1])_{M[1]}^{n', L}|}{|\mathrm{Aut}Y|}\cdot
\frac{\{Y,L'[1]\}}{\{Y,Y\}}=
\frac{|s_2\circ q_{12}^{-1}\circ s_1^{-1}(f)|}{|\mathrm{Aut}M|}\cdot
\frac{\{L',M\}}{\{M,M\}}.
$$
The above four identities induce the equivalence of \textbf{the symmetry-I} and \textbf{the symmetry-II}.

Here, we sketch the proof of Theorem \ref{maintheorem123}. To prove
$u_{[Z]}*(u_{[X]}*u_{[Y]})=(u_{[Z]}*u_{[X]})*u_{[Y]}$ is equivalent
to prove
$$
\sum_{[L]}F_{XY}^{L}F_{ZL}^{M}=\sum_{[L']}F_{ZX}^{L'}F_{L'Y}^{M}.
$$
IWe can check directly the following:
$$\mathrm{LHS}= \frac{1}{|\mathrm{Aut}
X|\cdot\{X,X\}}\sum_{[L]}\sum_{[L']}\frac{|\mathrm{Hom}(M\oplus
X,L)^{Y,Z[1]}_{L'[1]}|}{|\mathrm{Aut} L|}\cdot \frac{\{M\oplus
X,L\}}{\{L,L\}},
$$
and
$$
\mathrm{RHS}=\frac{1}{|\mathrm{Aut}
X|\cdot\{X,X\}}\sum_{[L']}\sum_{[L]}\frac{|\mathrm{Hom}(L',M\oplus
X)^{Y,Z[1]}_{L}|}{|\mathrm{Aut} L'|}\cdot \frac{\{L',M\oplus
X\}}{\{L',L'\}}.
$$
\textbf{The symmetry-I} naturally deduces LHS=RHS and then
the proof of Theorem \ref{maintheorem123} is obtained immediately.

By using $\textbf{the symmetry-II}$, we can provide a direct proof of the associativity of
$\mathcal{H}^{Dr}(\mathcal{C})$ .

\begin{theorem}\cite[Proposition 6.10]{KS}\label{KSalgebras}
The algebra $\mathcal{H}^{Dr}(\mathcal{C})$  is
associative.
\end{theorem}
\begin{proof}
For any $X, Y$ and $Z$ in $\mc$, we need to prove
$$v_{[Z]}*(v_{[X]}*v_{[Y]})=(v_{[Z]}*v_{[X]})*v_{[Y]}.$$ By the
definition of the multiplication, it is equivalent to prove
$$
\sum_{[L]}\{Y,X[1]\}\{L,Z[1]\}|\mathrm{Hom}(Y, X[1])_{L[1]}|\cdot
|\mathrm{Hom}(L, Z[1])_{M[1]}|
$$
$$
\hspace{-0.5cm}=\sum_{[L']}\{X,Z[1]\}\{Y,L'[1]\}|\mathrm{Hom}(X,
Z[1])_{L'[1]}|\cdot|\mathrm{Hom}(Y, L'[1])_{M[1]}|.
$$
Following the first statement of \textbf{the symmetry-II}, LHS is equal to
$$
\sum_{[L],[L']}\{Y,X[1]\}\{L,Z[1]\}|\mathrm{Hom}(Y,
X[1])_{L[1]}^{L'}|\cdot |\mathrm{Hom}(X,
Z[1])_{L'[1]}^{L}|\cdot$$$$|\mathrm{Hom}(Y, Z[1])|\cdot \{X\oplus Y,
Z[1]\}\cdot \{L, Z[1]\}^{-1}.
$$
Following the second statement of \textbf{the symmetry-II},  RHS is equal to
$$
\sum_{[L'],[L]}\{X,Z[1]\}\{Y,L'[1]\}|\mathrm{Hom}(X,
Z[1])_{L'[1]}^{L}|\cdot|\mathrm{Hom}(Y, X[1])_{L[1]}^{L'}|\cdot
$$
$$
|\mathrm{Hom}(Y, Z[1])|\cdot \{Y, X[1]\oplus Z[1]\}\cdot\{Y,
L'[1]\}^{-1}.
$$
The equality LHS=RHS is just the third statement of \textbf{the
symmetry-II}.
\end{proof}

\section{Motivic Hall algebras}
Let $\mathbb{K}$ be an algebraically closed field. An ind-constructible set is a countable union of non-intersecting
constructible sets.
\begin{example}\cite{JSZ2005, XXZ2006}
Let  $\bbc$ be the complex field and $A$ a finite dimensional
algebra $\bbc Q/I$ with indecomposable projective modules $P_i$,
$i=1, \cdots, l$. Given a projective complex
$P^{\bullet}=(P^i,\partial_i)_{i\in \mathbb{Z}}$ with
$P^i=\bigoplus_{j=1}^l e_j^i P_j.$ We denote by $\ue^i$ the vector
$(e^i_1,e^i_2,\cdots,e^i_l).$ The sequence, denoted by
 $\uee=\uee(P^{\bullet})=(\ue^{i})_{i\in \mathbb{Z}}$,  is called the projective dimension sequence of
$P^{\bullet}.$ We assume that only finitely many $\ue^i$ in $\uee$
are nonzero. Define $\mathcal{P}(A,\uee)$ to be the subset of
$$\prod_{i\in\bbz}\Hom_A(P^i,P^{i+1})=\prod_{i\in\bbz}\Hom_A(\bigoplus_{j=1}^l e_j^i P_j,\bigoplus_{j=1}^l e_j^{i+1}P_j)$$
which consists of elements $(\partial_i:P^i\ra P^{i+1})_{i\in\bbz}$
such that $\partial_{i+1}\partial_i=0$ for all $i\in\bbz.$  It is an
affine variety with a natural action of the algebraic group
$G_{\uee}=\prod_{i\in\bbz}\aut_A(P^i)$. Let $K_0(\md^b(A)),$ or
simply $K_0,$ be the Grothendieck group of the derived category
$\md^b(A),$ and $\udim:\md^b(A)\ra K_0(\md^b(A))$ the canonical
surjection. It induces a canonical surjection from the abelian group
of dimension vector sequences to $K_0,$ we still denote it by
$\udim.$ Given $\bd\in K_0$, the set
$$\mathcal{P}(A,\bd)=\bigsqcup_{\uee\in\udim^{-1}(\bd)}\mathcal{P}(A,\uee)$$ is an
ind-constructible set.
\end{example}

We recall the notion of motivic invariants of quasiprojecitve varieties in \cite{Joyce2} (see also \cite{Bri2010}). Suppose $\Lambda$ is a commutative $\bbq$-algebra with identity $1$. Let $\Upsilon: \{\mbox{isomorphism classes }[X] \mbox{ of quasiprojective }\mathbb{K}-\mbox{varieties }X\}\rightarrow \Lambda $ satisfy that:
\begin{enumerate}
  \item $\Upsilon([X])=\Upsilon([Z])+\Upsilon([U])$ for a closed subvariety $Z\subseteq X$ and $U=X\setminus Z$;
  \item $\Upsilon([X\times Y])=\Upsilon([X])\Upsilon([Y])$;
  \item Write $\Upsilon([\mathbb{K}])=\mathbb{L}$. Then $\mathbb{L}$ and $\mathbb{L}^k-1$ for $k=1,2,\cdots$ are invertible in $\Lambda.$
\end{enumerate}

Let $X$ be a constructible set over $\mathbb{K}$ and $G$ an affine algebraic
group acting on $X$. Then $(X, G)=(X, G, \alpha)$ is called a constructible stack \cite[Section4.2]{KS}, where $\alpha$ is an action of $G$ on $X$. In \cite[Section 4.2]{KS}, the authors defined the 2-category of constructible stacks. Define $Mot_{st}((X, G))=Mot_{st}((X, G), (\Upsilon, \Lambda))$ to be the $\Lambda$-module generated by equivalence classes of $1$-morphisms of constructible stacks $[(Y, H)\rightarrow (X, G)]$ with the following relations:
\begin{enumerate}
  \item $[((Y_1, G_1)\bigsqcup (Y_2, G_2)\rightarrow (X, G)]=[(Y_1, G_1)\rightarrow (X, G)]+[(Y_2, G_2)\rightarrow (X, G)]$;
  \item $[(Y, H)\rightarrow (X, G)]=[(Z\times A_{\mathbb{K}}^d, H)\rightarrow (X, G)]$ if $Y\rightarrow Z$ is an $H$-equivariant constructible vector bundle of rank $d$;
  \item Let $(Y, H)$ be a constructible stack and $U$ a quasiprojective $\mathbb{K}$-variety with trivial action of $H$.  Let $\pi: (Y, H)\times U\rightarrow (Y, H)$. Then $[\rho: (Y, H)\times U\rightarrow (X, G)]=\Upsilon(U)[\rho\circ\pi: (Y, H)\rightarrow (X, G)].$
\end{enumerate}

Let us  recall the definition of motivic Hall algebras in \cite{KS}. Let $\mc$ be an ind-constructible triangulated $A_{\infty}$-category over $\mathbb{K}$ and objects in $\mathcal{C}$ form an ind-constructible set
$\mathfrak{Obj}(\mathcal{C})=\bigsqcup_{i\in I}\mathcal{X}_i$ for
countable constructible sets $\mathcal{X}_i$ with the action of an
affine algebraic group $G_i$ on $\mathcal{X}_i.$
For $X, Y\in \mathfrak{Obj}(\mathcal{C})$,  set
$$
\{X,
Y\}:=\mathbb{L}^{\sum_{i>0}(-1)^i \mathrm{dim}_{\mathbb{C}}
\mathrm{Hom}(X[i], Y)}.
$$ For any $i, j\in I$, consider the maps $\Phi_1: \mathcal{X}_i\times \mathcal{X}_j\rightarrow \Lambda$ sending $(M, N)$ to $$\Upsilon([\mathrm{Hom}_{\mc}(M, N)])=\mathbb{L}^{\mathrm{dim}_{\mathbb{K}}\mathrm{Hom}_{\mc}(M, N)},$$ $\Phi_2: \mathcal{X}_i\times \mathcal{X}_j\rightarrow \Lambda$ mapping $(M, N)$ to $\{M, N\}$ and $\Phi_3: \mathcal{X}_i\rightarrow \Lambda$ sending $M$ to $\Upsilon([\mathrm{Aut}(M)])$ (see Proposition \ref{autogroup}). In the following, we assume that all three maps are constructible functions.

Consider the $\Lambda$-module  $$\mathcal{MH}(\mathcal{C})=\bigoplus_{i\in
I}Mot_{st}((\mathcal{X}_i, G_{i})).$$
One can endow the module with the multiplication
$$
[\pi_1: \mathcal{S}_1\rightarrow \mathfrak{Obj}(\mathcal{C})]*
[\pi_2: \mathcal{S}_2\rightarrow \mathfrak{Obj}(\mathcal{C})]=\sum_{n\in \bbz}\mathbb{L}^{-n}[\pi:
\mathcal{W}_n\rightarrow \mathfrak{Obj}(\mathcal{C})],
$$
where $\pi_1(\mathcal{S}_1)\subseteq \mathcal{X}_i$, $\pi_2(\mathcal{S}_2)\subseteq \mathcal{X}_j$ for some $i, j\in I$ and$$\mathcal{W}_n=\{(s_1, s_2, \alpha)\mid s_i\in \mathcal{S}_i,
\alpha\in \mathrm{Hom}_{\mathcal{C}}(\pi_2(s_2), \pi_1(s_1)[1]),$$$$
\sum_{i>0}(-1)^i \mathrm{dim}_{\mathbb{C}}
\mathrm{Hom}(\pi_2(s_2)[i], \pi_1(s_1)[1])=-n\}.$$ The map $\pi$
sends $(s_1, s_2, \alpha)$ to $Cone(\alpha)[-1].$ Here, for simplicity of notations, we write $[S\rightarrow \mathcal{X}_i]$ instead of $[(S, G_s)\rightarrow (\mathcal{X}_i, G_i)]$. The algebra $\mathcal{MH}(\mathcal{C})$ is called the motivic Hall algebra associated to $\mc.$

For convenience, we use the integral notation for the right-hand term in the definition of the multiplication. Then  the multiplication can be rewritten as
$$
[\pi_1: \mathcal{S}_1\rightarrow \mathfrak{Obj}(\mathcal{C})]\cdot
[\pi_2: \mathcal{S}_2\rightarrow \mathfrak{Obj}(\mathcal{C})]$$$$:=\int_{s_1\in \mathcal{S}_1, s_2\in \mathcal{S}_2}[\mathrm{Hom}_{\mathcal{C}}(\pi_2(s_2), \pi_1(s_1)[1])\rightarrow \mathfrak{Obj}(\mathcal{C})]\cdot\{\pi_2(s_2), \pi_1(s_1)[1]\}
$$
$$
\hspace{-1.1cm}:=\int_{s_1\in \mathcal{S}_1, s_2\in \mathcal{S}_2}\{\pi_2(s_2), \pi_1(s_1)[1]\}\cdot\int_{\alpha\in \mathrm{Hom}_{\mathcal{C}}(\pi_2(s_2), \pi_1(s_1)[1])_{E}}v_{[E]}
$$
where $v_{[E]}:=[\pi: pt\rightarrow
\mathfrak{Obj}(\mathcal{C})]$ with $\pi(pt)=E.$ Note that $\mathrm{Hom}_{\mathcal{C}}(\pi_2(s_2), \pi_1(s_1)[1])_{E}$ is a constructible set \cite[Appendix]{CK}.
\begin{theorem}\cite[Proposition 10]{KS}
With the above multiplication, $\mathcal{MH}(\mathcal{C})$ becomes
an associative algebra.
\end{theorem}
Inspired by \cite{KS} and \cite{XX2006}, the proof can be considered as a motivic
version of \textbf{the symmetry-II}.
\begin{proof}By the reformulation of the definition of  multiplication, the proof of the theorem is easily reduced to the case when
$\mathcal{S}_i$ is just a point. Given $X, Y$ and $Z\in \mathfrak{Obj}(\mathcal{C})$, $v_{[Z]}*(v_{[X]}*v_{[Y]})$ is equal to
$$
\mathcal{T}_1:=\int_{\alpha\in \mathrm{Hom}(Y,
X[1])_{L[1]}}\int_{\beta\in \mathrm{Hom}(L, Z[1])_{M[1]}}\{Y,X[1]\}\cdot\{L,Z[1]\}\cdot v_{[M]}
$$
and $(v_{[Z]}*v_{[X]})*v_{[Y]}$ is equal to
$$
\mathcal{T}_2:=\int_{\alpha'\in \mathrm{Hom}(X,
Z[1])_{L'[1]}}\int_{\beta'\in \mathrm{Hom}(Y, L'[1])_{M[1]}}\{X,Z[1]\}\cdot\{Y,L'[1]\}\cdot v_{[M]}.
$$
Using the notation in Section $3$, we have
$$
\mathcal{T}_1=\int_{\alpha'\in \mathrm{Hom}(X, Z[1])^L_{L'[1]}, \alpha\in \mathrm{Hom}(Y,
X[1])^{L'}_{L[1]}}\int_{\beta\in \mathrm{Hom}(L, Z[1])^{\alpha, L'}_{M[1]}}\{Y,X[1]\}\cdot\{L,Z[1]\}\cdot v_{[M]}
$$
and
$$
\mathcal{T}_2=\int_{\alpha\in \mathrm{Hom}(Y,
X[1])^{L'}_{L[1]}, \alpha'\in \mathrm{Hom}(X,
Z[1])^L_{L'[1]}}\int_{\beta'\in \mathrm{Hom}(Y, L'[1])^{\alpha', L}_{M[1]}}\{X,Z[1]\}\cdot\{Y,L'[1]\}\cdot v_{[M]}.
$$
As  in the proof of Theorem \ref{KSalgebras}, fix $\alpha\in
\mathrm{Hom}(Y, X[1])_{L[1]}^{L'}$, by Diagram
\ref{maindiagramme}, there is a constructible bundle
$\mathrm{Hom}(L, Z[1])_{M[1]}^{\alpha, L'}\rightarrow
\mathrm{Hom}(X, Z[1])_{L'[1]}^{L}$ with fibre dimension
$$
\mathrm{dim}\mathrm{Hom}(Y, Z[1])+\mathrm{dim}\{X\oplus Y,
Z[1]\}-\mathrm{dim}\{L, Z[1]\}.
$$
Fix $\alpha'\in \mathrm{Hom}(X, Z[1])^L_{L'[1]}$, by Diagram
\ref{maindiagramme}, there is a constructible bundle
$$\mathrm{Hom}(Y, L'[1])_{M[1]}^{\alpha', L}\rightarrow
\mathrm{Hom}(Y, X[1])_{L[1]}^{L'}$$ with fibre dimension
$$
\mathrm{dim}\mathrm{Hom}(Y, Z[1])+\mathrm{dim}\{Y, X[1]\oplus
Z[1]\}-\mathrm{dim}\{Y, L'[1]\}.
$$
Hence, we have $\mathcal{T}_1=\mathcal{T}_2.$
\end{proof}

Given an indecomposable object $X\in \mc$, $\Upsilon([\mathrm{End}_{\mc}(X)])=\mathbb{L}^{\mathrm{dim}_k\mathrm{End}_{\mc}(X)}$ and
$$\Upsilon([\mathrm{Aut}X])=\mathbb{L}^{\mathrm{dim}_k\mathrm{radEnd}X}(\mathbb{L}^{d(X)}-1)$$
where $d(X)=\mathrm{dim}_k(\mathrm{End}X/\mathrm{radEnd}X).$ Given $n\in \bbn$, consider the morphism $$\mathrm{Aut}(nX)\rightarrow \mathrm{GL}_n(\mathrm{End}X/\mathrm{radEnd}X),$$ the fibre is an affine space (consisting of matrices with elements belonging to $\mathrm{radEnd}X$) of dimension $n^2\mathrm{dim}_k(\mathrm{radEnd}X)$.  Hence, by \cite[Lemma 2.6]{Bri2010}, we have
$$
\Upsilon([\mathrm{Aut}(nX)])=\mathbb{L}^{n^2\mathrm{dim}_k(\mathrm{radEnd}X)+\frac{1}{2}n(n-1)d(X)}\prod_{k=1}^n(\mathbb{L}^{kd(X)}-1).
$$
Generally, an object $X\in \mc$ is isomorphic to $n_1X_1\oplus n_2X_2\oplus \cdots \oplus n_tX_t$ where $X_i\ncong X_j$ for $i, j=1, \cdots, t$ and $i\neq j.$
Consider the natural morphism $$\mathrm{Aut}(X)\rightarrow \mathrm{Aut}(n_1X_1)\times \cdots \times  \mathrm{Aut}(n_tX_t).$$ It is a vector bundle of dimension $\sum_{i\neq j}\mathrm{Hom}(n_iX_i, n_jX_j)$. Hence, we have the following result.
\begin{Prop}\label{autogroup}
For $X\in \mc$, $\Upsilon([\mathrm{Aut}X])=\mathbb{L}^t\prod_{i=1}^l\prod_{j=1}^{s_i}(\mathbb{L}^{jd(X)}-1)$  for some $t, l, s_1, \cdots s_l\in \bbn$ and then is invertible in $\Lambda.$
\end{Prop}
We introduce some necessary notations. Let $\mathcal{W}=\bigsqcup_{n\in \bbz}\mathcal{W}_n$. Then $\pi:\mathcal{W}\rightarrow \mathfrak{Obj}(\mathcal{C})$ (for simplicity, we use the same notation as $\pi:\mathcal{W}_n\rightarrow \mathfrak{Obj}(\mathcal{C})$) induces that $\pi(\mathcal{W})\subseteq \mathfrak{Obj}(\mathcal{C})$. For $X\in \mc$, if $\Upsilon([\mathrm{Aut}X])=\mathbb{L}^t\prod_{i=1}^l\prod_{j=1}^{s_i}(\mathbb{L}^{jd(X)}-1)$, then write $d_a(X)=(t, s_1, s_2, \cdots, s_l)$ and $\mathbb{L}^{d_a}=\mathbb{L}^t\prod_{i=1}^l\prod_{j=1}^{s_i}(\mathbb{L}^{jd(X)}-1).$ For $X, Y, L\in \mc$, write $$d_{\{X, Y\}}=\sum_{i>0}(-1)^i \mathrm{dim}_{\mathbb{C}}
\mathrm{Hom}(X[i], Y)$$
and $d^*=d_{(X, Y, L)}=(d_a(X), d_{\{X, L\}}, d_{\{X, X\}}).$ For a fixed triple $d^*=(d_a, l, m)$ with $d_a=(t, s_1, s_2, \cdots, s_l)$ and two pairs $[\pi_i: \mathcal{S}_i\rightarrow \mathfrak{Obj}(\mathcal{C})]$ for $i=1, 2$, define
$$
\mathcal{V}_{d^*}=\{(s_1, s_2, L, \beta)\mid L\in \pi(\mathcal{W}), \beta\in \mathrm{Hom}(\pi_1(s_1), L) \mbox{ with } \mathrm{Cone}(\beta)\cong \pi_2(s_2) \mbox{ and }$$
$$d_a(\mathrm{Aut}(\pi_1(s_1)))=d_a, d_{\{\pi_1(s_1), Cone(\alpha)\}}=l, d_{\{\pi_1(s_1), \pi_1(s_1)\}}=m.\}
$$
Consider the $\Lambda$-module  $$\mathcal{MH}_T(\mathcal{C})=\bigoplus_{i\in
I}Mot_{st}((\mathcal{X}_i, G_{i}))$$ endowed with the multiplication
$$[\pi_1: \mathcal{S}_1\rightarrow \mathfrak{Obj}(\mathcal{C})]\cdot
[\pi_2: \mathcal{S}_2\rightarrow \mathfrak{Obj}(\mathcal{C})]=\sum_{d_a, l, m}[\psi:\mathcal{V}_{d^{*}}\rightarrow \mathfrak{Obj}(\mathcal{C}) ]\mathbb{L}^{-d_a}\mathbb{L}^{l-m}$$
$$
:=\int_{s_1\in \mathcal{S}_1, s_2\in \mathcal{S}_2}\Upsilon([\mathrm{Aut}(\pi(s_1))])^{-1}\{\pi_1(s_1), \pi_1(s_1)\}^{-1}\int_{L\in \pi(\mathcal{W})}[\mathrm{Hom}(\pi_1(s_1), L)\rightarrow \mathfrak{Obj}(\mathcal{C}) ]
$$
where $\pi_1(\mathcal{S}_1)\subseteq \mathcal{X}_i$, $\pi_2(\mathcal{S}_2)\subseteq \mathcal{X}_j$ for some $i, j\in I$ and $\psi(s_1,s_2,L,\beta)=L.$
Then $\mathcal{MH}_T(\mathcal{C})$ is an $\Lambda$-algebra.

Given $Z, M\in \mc$ and $l:Z\rightarrow M$, there is a unique diatinguished triangle
$$
\xymatrix{Z\ar[r]^l&M\ar[r]^m&L\ar[r]^n&Z[1]}
$$
where $L=Cone(l)$ and $m=\left(
                           \begin{array}{cc}
                             1 & 0 \\
                           \end{array}
                         \right)
$, $n=\left(
        \begin{array}{c}
          0 \\
          1 \\
        \end{array}
      \right).
$
Set $$n\mathrm{Hom}(Z[1], L)=\{nt\mid t\in \mathrm{Hom}(Z[1], L)\} \mbox{ and } \mathrm{Hom}(Z[1], L)n=\{tn\mid t\in \mathrm{Hom}(Z[1], L)\}.$$
It is easy to check that they are vector spaces.
\begin{lemma}\cite[Lemma 2.4]{XX2006}\label{space}
With the above notation, we have
$$
\Upsilon([n\mathrm{Hom}(Z[1], L)])=\{M, L\}\{Z, L\}^{-1}\{L, L\}^{-1}$$ and $$\Upsilon([\mathrm{Hom}(Z[1], L)n])=\{Z, M\}\{Z, L\}^{-1}\{Z, Z\}^{-1}.
$$
\end{lemma}
Now we give the motivic version of the derived Riedtmann-Peng formula.
\begin{Prop}\label{motivicRP}
For $Z, L, M\in \mc$, we have
$$
[\mathrm{Hom}_{\mc}(Z, M)_L\rightarrow \mathfrak{Obj}(\mathcal{C})]\cdot\frac{\{Z, M\}}{\Upsilon([\mathrm{Aut}Z])\cdot\{Z, L\}\cdot\{Z, Z\}}
$$
$$
=[\mathrm{Hom}_{\mc}(M, L)_{Z[1]}\rightarrow \mathfrak{Obj}(\mathcal{C})]\cdot\frac{\{M, L\}}{\Upsilon([\mathrm{Aut}L])\cdot\{Z, L\}\cdot\{L, L\}}.
$$
\end{Prop}
\begin{proof}
Define the constructible set
$$S_1=\{(a_L, l, tn)\mid a_L\in\mathrm{Aut}L, l\in \mathrm{Hom}_{\mc}(Z, M)_L, t\in \mathrm{Hom}_{\mc}(Z[1], L)\}$$
and
$$
S_2=\{(a_Z, m, nt)\mid a_Z\in \mathrm{Aut}Z, m\in \mathrm{Hom}_{\mc}(M, L)_{Z[1]}, g_2\in \mathrm{Aut}L, t\in \mathrm{Hom}_{\mc}(Z[1], L)\}.
$$
Here, the projection $S_1\rightarrow \mathrm{Aut}L\times \mathrm{Hom}_{\mc}(Z, M)_L$ is a constructible bundle of dimension $\mathrm{dim}_k\mathrm{Hom}_{\mc}(Z[1], L)n$ (it is irrelevant to the choice of $n$). In the same way, $S_2$ is a constructible bundle of dimension $n\mathrm{dim}_k\mathrm{Hom}_{\mc}(Z[1], L)$. Note that given $g_1\in \mathrm{Aut}Z$, $g_1\circ l=l$ means  that $g_1=1+g'_1$ and $g'_1\circ l=0$. However, $g'_1\circ l=0$ is equivalent to say $g'_1=t\circ n$ for some $t\in \mathrm{Hom}(Z[1], L)$. Similarly, $g_2\in \mathrm{Aut}L$ satisfies $mg_2=m$ if and only if $g_2\in 1+n\mathrm{Hom}(Z[1], L).$ For any $l'\in \mathrm{Hom}_{\mc}(Z, M)_L$, there exists unique $a_Z\in \mathrm{Aut}Z$ such that $a_Zl=l'.$ There is an isomorphism $S_1\rightarrow S_2$ defined by sending $(a_L, l', tn)$ to $(a_Z,a_Lm, nt).$
Hence, we have $[S_1\rightarrow \mathfrak{Obj}(\mathcal{C})]=[S_2\rightarrow \mathfrak{Obj}(\mathcal{C})].$ On the other hand, by the definition of motivic Hall algebras and Lemma \ref{space}, we have $$
[S_1\rightarrow \mathfrak{Obj}(\mathcal{C})]=\Upsilon([\mathrm{Aut}L])\{Z, M\}\{Z, L\}^{-1}\{Z, Z\}^{-1}\cdot[\mathrm{Hom}_{\mc}(Z, M)_L\rightarrow \mathfrak{Obj}(\mathcal{C})]
$$
and
$$
[S_2\rightarrow \mathfrak{Obj}(\mathcal{C})]=\Upsilon([\mathrm{Aut}Z])\{M, L\}\{Z, L\}^{-1}\{L, L\}^{-1}\cdot[\mathrm{Hom}_{\mc}(M, L)_{Z[1]}\rightarrow \mathfrak{Obj}(\mathcal{C})].
$$
This concludes the proof of the proposition.
\end{proof}
\begin{theorem}
With the above defined multiplication, $\mathcal{MH}_{T}(\mathcal{C})$ becomes
an associative algebra.
\end{theorem}
The proof  can be considered as the motivic version of the proof for \cite[Theorem 3.6]{XX2006}.
\begin{proof}
Set $u_{[E]}:=[\pi: pt\rightarrow \mathfrak{Obj}(\mathcal{C})]$ with $\pi(pt)=E.$ Given three functions $[\pi_i: \mathcal{S}_i\rightarrow \mathfrak{Obj}(\mathcal{C})]$ for $i=1,2,3$, we need to prove
$$
[\pi_3]\cdot([\pi_1]\cdot[\pi_2])=([\pi_3]\cdot[\pi_1])\cdot[\pi_2].
$$ By the reformulation of the  definition of multiplication, the proof of the theorem is easily reduced to the case when
$\mathcal{S}_i$ is just a point. Let $\pi_3(pt)=Z, \pi_2(pt)=Y$ and $\pi_1(pt)=X$. Set $t_{[X]}=\Upsilon^{-1}([\mathrm{Aut}(X)])\cdot \{X, X\}$. Then $u_{[Z]}*(u_{[X]}*u_{[Y]})$ is equal to
$$
\int_{L\in \pi(\mathcal{W}), L'\in \pi(\mathcal{W}')}[\mathrm{Hom}(M\oplus X, L)_{L'[1]}^{Y, Z[1]}\rightarrow \mathfrak{Obj}(\mathcal{C})]t_{[X]}t_{[L]}\{M\oplus X, L\}.
$$
where $\pi(\mathcal{W}')$ is the image of $\pi: \mathrm{Ext}^1(Z, X)\rightarrow  \mathfrak{Obj}(\mathcal{C})$ by sending $\alpha$ to its middle term.
Similarly,
we have that
$(u_{[Z]}*u_{[X]})*u_{[Y]}$ is equal to
$$
\int_{L\in \pi(\mathcal{W}), L'\in \pi(\mathcal{W}')}[\mathrm{Hom}(L', M\oplus X)_{L}^{Y, Z[1]}\rightarrow \mathfrak{Obj}(\mathcal{C})]t_{[X]}t_{[L']}\{L', M\oplus X\}.
$$
Following Proposition \ref{motivicRP}, $(u_{[Z]}*u_{[X]})*u_{[Y]}=u_{[Z]}*(u_{[X]}*u_{[Y]}).$
\end{proof}

\begin{theorem}
There exists an algebra isomorphism $\Phi: \mathcal{MH}(\mathcal{C})\rightarrow \mathcal{MH}_{T}(\mathcal{C})$ defined by $$\Phi([\pi: \mathcal{S}\rightarrow \mathfrak{Obj}(\mathcal{C})])=\int_{s\in \mathcal{S}}t_{[\pi(s)]}u_{[\pi(s)]}$$ where $t_{[\pi(s)]}=\Upsilon^{-1}([\mathrm{Aut}(\pi(s))])\cdot \{\pi(s), \pi(s)\}.$
\end{theorem}
The proof is a direct application of Proposition \ref{motivicRP}.

\end{document}